\documentclass{amsart}
\usepackage[utf8]{inputenc}

\usepackage{graphicx}
\usepackage{tikz-cd}
\usepackage{hyperref}
\usepackage{cite}
\usepackage{todonotes}

\newtheorem{theorem}{Theorem}[section]
\newtheorem{theoremA}{Theorem}
\newtheorem{lemma}[theorem]{Lemma}

\theoremstyle{definition}

\theoremstyle{remark}
\newtheorem{remark}[theorem]{Remark}

\numberwithin{equation}{section}

\newcommand{\supp}{\text{supp}\,}

\newcommand{\Isom}{\text{Isom}\,}
\newcommand{\nin}{\noindent}

\begin{document}

\title[Equilibrium measures for certain isometric extensions]{Equilibrium measures for certain isometric extensions of Anosov systems\\
}

\author{Ralf Spatzier}
%\address{Department of Mathematics, University of Michigan, Ann Arbor, MI}
%\email{spatzier@umich.edu}
\thanks{The first author was supported in part by NSF Grant DMS--1307164 and a research professorship at MSRI}

\author{Daniel Visscher}
%\address{Department of Mathematics, University of Michigan, Ann Arbor, MI}
%\email{davissch@umich.edu}
\thanks{The second author was supported by NSF RTG grant 1045119}

%    General info
%\subjclass[2000]{Primary 54C40, 14E20; Secondary 46E25, 20C20}

\date{}

%\dedicatory{}

%\keywords{}

\begin{abstract}
We prove that for the frame flow on a negatively curved, closed manifold of odd dimension other than 7, and a H\"older continuous potential that is constant on fibers, there is a unique equilibrium measure. We prove  a similar result for automorphisms of the Heisenberg manifold fibering over the torus.  Our methods also give an alternate proof of Brin and Gromov's result on the ergodicity of these frame flows. 
\end{abstract}

\maketitle

\section{Introduction}

\textit{Topological entropy} is a measure of the complexity of a dynamical system on a compact topological space. It records the exponential growth rate of the amount of information needed to capture the system for time $t$ at a fine resolution $\epsilon$, as $t \rightarrow \infty$ and $\epsilon \rightarrow 0$. As a topological invariant, it can be used to distinguish between topologically different dynamical systems. Positive topological entropy is also used as an indicator of chaos. One may wish to be selective of the information in the system to include, however; such a selection can be encoded by a probability measure. Given an invariant probability measure, \textit{measure theoretic entropy} computes the complexity of the dynamical system as seen by the measure. This, of course, depends on the measure chosen. These two ways of computing entropy are related by the \textit{variational principle}, which states that the topological entropy is the supremum of measure-theoretic entropies over the set of $f$-invariant measures. The variational principle provides a tool for picking out distinguished measures---namely, those that maximize the measure-theoretic entropy. Such measures (if they exist) are called \textit{measures of maximal entropy}.

\textit{Pressure} is a generalization of entropy which takes into account a weighting of the contribution of each orbit to the entropy by a H\"older continuous potential function. In the case that the potential is identically zero, the pressure is just the entropy of the system. The variational principle also applies to topological and measure theoretic pressure and implies that for a given dynamical system, any H\"older continuous potential function determines a set (possibly empty) of invariant measures that maximize the measure theoretic pressure. Such measures are called \textit{equilibrium measures}. As is well known, work of Newhouse and Yomdin shows that equilibrium measures for continuous functions and $C^\infty$ dynamics always exists \cite{newhouse89}.

In the 1970's, Bowen and Ruelle produced a set of results considering entropy, pressure, and equilibrium measures for Axiom A and, in particular, Anosov diffeomorphisms and flows~\cite{Bowen, Bowen-Ruelle}. A central result is the following: given a transitive Anosov diffeomorphism or flow of a compact manifold and a H\"older continuous potential function, there exists an equilibrium measure and it is unique. Moreover, this measure is ergodic and has full support. While this theorem applies broadly to Anosov dynamics, there is no general theory for partially hyperbolic systems. Results are limited to specific sets of diffeomorphisms and potentials (quite often only the zero potential). In the main part of this paper, we study equilibrium measures for the (full) frame flow $F^t$ on a negatively curved, closed manifold $M$ and a particular class of potentials. Recall that $F^t$ is a flow on the positively oriented orthonormal frame bundle $FM$, which factors over the unit tangent bundle $SM$; see Section~\ref{sec:prelim} for definitions. When the dimension of the underlying manifold $M$ is at least $3$, such frame flows are (non-Anosov) partially hyperbolic flows. Indeed, the orthonormal frame bundle fibers non-trivially over the unit tangent bundle if 
$\dim (M) \geq 3$, and the frame flow is isometric along the fibers. This isometric behavior along a foliation will play an important role in proving the following result.
\begin{theoremA}\label{main}
Let $M$ be a closed, oriented, negatively curved $n$-manifold, with $n$ odd and not equal to $7$. For any H\"older continuous potential $\varphi: FM \rightarrow \mathbb{R}$ that is constant on the fibers of the bundle $FM \rightarrow SM$, there is a unique equilibrium measure for $(F^t,\varphi)$. It is ergodic and has full support.
\end{theoremA}

Let us make a couple of comments on the assumptions of this theorem. First, the restriction on the dimension of the manifold in this theorem is due to a topological argument using the non-existence of certain transitive actions on spheres. Second, while the condition that the potential function is constant on the fibers is highly restrictive, it does apply to any H\"older function pulled back from a function on the unit tangent bundle. In particular, the theorem applies to the measure of maximal entropy, which is the equilibrium state for the zero potential, or equivalently in this case, for the unstable Jacobian. Our assumption on the potential makes the measure amenable to the methods used in the proof (namely, the projected measure has local product structure). We believe that the theorem should hold for a more general class of functions, but the problem becomes much more difficult.

The methods of the proof also apply in other situations,  for instance to certain automorphisms of a nilmanifolds 
that factor over an Anosov map. Here is a sample result. Let $Heis$ be the  3-dimensional Heisenberg group, and $Heis(\mathbb{Z})$ its integer lattice.  Let $M = Heis/ Heis (\mathbb{Z})$ be the Heisenberg manifold. Note that $M$ naturally fibers over the 2-torus  $\mathbb{T}^2$ by factoring by the center of $Heis$.  We hope to pursue these matters in greater detail in the future.  

\begin{theoremA}   \label{heisenberg}
Let $M$ be the Heisenberg manifold, and 
 $f$  a partially hyperbolic automorphism of $M$ such that the induced action on the base torus is Anosov and the action on the fibers is isometric.  
 Then for any H\"older continuous potential $\varphi: M \rightarrow \mathbb{R}$ that is constant on the fibers of the canonical projection map $M \to \mathbb{T}^2$, there is a unique equilibrium measure for $(f, \varphi)$. It is ergodic and has full support.
\end{theoremA}

We remark that the equilibrium states of Theorems~\ref{main} or \ref{heisenberg} are in one-to-one correspondence with cohomology classes in the set of potential functions constant on fibers. Indeed, the equilibrium measure is uniquely determined by an equilibrium measure for the Anosov base dynamics, where this is a classical result.

Recent study of the existence and uniqueness of equilibrium measures for partially hyperbolic diffeomorphisms and flows focuses on examples. The following is a list of results most pertinent to the present paper. The existence and uniqueness of a measure of maximal entropy for diffeomorphisms of the 3-torus homotopic to a hyperbolic automorphism was shown by Ures in~\cite{Ures}. Rodriguez-Hertz, Rodriguez-Hertz, Tahzibi, and Ures proved existence and uniqueness of the measure of maximal entropy for 3-dimensional partially hyperbolic diffeomorphisms with compact center leaves when the central Lyapunov exponent is zero, and multiple measures of maximal entropy when the central Lyapunov exponent is non-zero~\cite{RHRHTU12}. Climenhaga, Fisher, and Thompson showed existence and uniqueness of equilibrium measures under conditions on the potential function for certain derived-from-Anosov diffeomorphisms of tori~\cite{Climenhaga-Fisher-Thompson}. Finally, Knieper proved for geodesic flows in higher rank symmetric spaces that the measure of maximal entropy is again unique \cite{Knieper05} with support a submanifold of the unit tangent bundle on which the geodesic flow is partially  hyperbolic.  We note that he also proved uniqueness of the measure of maximal entropy for the geodesic flow on closed rank 1 manifolds of nonpositive curvature \cite{Knieper}.  These are non-uniformly hyperbolic but not usually partially hyperbolic.   

Bowen and Ruelle studied equilibrium measures for uniformly hyperbolic diffeomorphisms and flows via expansivity and specification. These results have been extended to weak versions of expansivity and specification by Climenhaga and Thompson in~\cite{Climenhaga-Thompson}, and used by them and Fisher in~\cite{Climenhaga-Fisher-Thompson}.
To our knowledge, outside of measures of maximal entropy, Theorems~\ref{main} and \ref{heisenberg} are the first results about the uniqueness of equilibrium measures for partially hyperbolic systems for which not even a weak form of specification holds.

The main method of proof is to combine ideas from measure rigidity of higher rank abelian actions  with ideas from the proof of Liv\v{s}ic' theorem on measurable cohomology of H\"older cocycles \cite{livsic}.  More precisely, we consider conditional measures along the central foliations.  The support of a conditional measure generates  limits of isometries along central leaves which act transitively on this support.  Thus, for frame flows, one arrives at a dichotomy for conditional measures: they are either invariant under the action of $SO(n-1)$ or else under a proper subgroup.  In the first case, the problem is reduced to understanding the equilibrium state projected to the unit tangent bundle where the projected flow is Anosov and classical methods apply.  In the second case, we get an invariant measurable section of an associated bundle.  As discussed in \cite{GS97}, the ideas of Liv\v{s}ic then show that such sections have to be continuous and even smooth, giving us a reduction of structure group of the frame bundle.  In the case that $n$ is odd and not $7$, the latter case gives a contradiction, as shown by Brin and Gromov.  Similar considerations and topological restrictions apply in the case of the Heisenberg manifold in Theorem \ref{heisenberg}.

The idea to study invariant measures via their conditional measures along isometric foliations was introduced in \cite{KS96}, and used repeatedly in other works (e.g., \cite{EKL06,EL15,LS04,AVW}). In particular, Lindenstrauss and Schmidt analyzed invariant measures for partially hyperbolic automorphisms of tori and more general compact abelian groups in \cite{LS04,LS05}. They showed that for measures singular with respect to Haar measure, the conditional measures along central foliations must be finite. 
We note that there are many such measures, however.
Indeed, their measure-theoretic entropies form a dense set in the interval between 0 and the topological entropy $h_{top} (f)$ \cite{Sun2012}, and they conjecturally fill out the entire interval $[0,h_{top}(f)]$.
Equilibrium measures are of course much more special. While the situation is classical and well understood for hyperbolic toral automorphisms, the non-expansive case is unclear: are equilibrium states unique for a given potential function?  How many equilibrium states are there in total?  
In contrast with general invariant measures for partially hyperbolic toral automorphisms, we find for the Heisenberg manifold and a specific family of potential functions that equilibrium measures are unique and are uniquely determined by the cohomology class of the potential function.  
Finally, let us remark that Avila, Viana, and Wilkinson studied conditional measures on center leaves and their invariance under stable and unstable holonomy in their work on measure rigidity of perturbations of the time 1 map of the geodesic flow of a hyperbolic surface \cite{AVW}.

In Section~\ref{sec:reproof}, we  adapt our methods to give a different, softer proof of the following classical result by Brin and Gromov on the ergodicity of certain frame flows. 

\begin{theoremA}[Brin-Gromov, \cite{Brin-Gromov}] \label{thm:Brin-Gromov}
Let $M$ be an odd dimensional closed, oriented manifold of negative sectional curvature and dimension $n \neq 7$. Then the frame flow is ergodic.  
\end{theoremA}

\noindent We remark that Brin and Karcher \cite{Brin-Karcher} proved ergodicity of the frame flow in even dimensions $\neq 8$ under pinching assumptions on the curvature. These results were extended under pinching restrictions on the curvature to dimensions $7$ and $8$ by Burns and Pollicott in \cite{Burns-Pollicott}. Our approach does not apply to such results, since those authors use pinching to control Brin-Pesin groups.

The novelty in our approach lies in replacing the Brin-Pesin group used in the works cited above by our groups $G_x^\mu$. While the important properties of the Brin-Pesin groups rely on the hyperbolicity of the geodesic flow, the relevant properties of $G_x^\mu$ rely only on the ergodicity of the geodesic flow. Our need for hyperbolicity appears in the use of a Liv\v{s}ic regularity theorem (Theorem~\ref{sectionregularity}) that requires that the dynamics on the base be Anosov.

\vspace{1em}

\noindent {\em Acknowledgements:}  We thank A. Wilkinson, V. Climenhaga, and T. Fisher for discussions about this project.

\section{Preliminaries} \label{sec:prelim}

We first review some basic definitions and results. 

\subsection{Frame flow} 
Let $M$ be a closed, oriented $n$-dimensional manifold with Riemannian metric. Let $SM = \{ (x,v) : x \in M, v \in T_xM, \| v \| = 1\}$ denote the unit tangent bundle, and let $FM = \{ (x;v_0,v_1,\ldots,v_{n-1}) : x \in M, v_i \in T_xM \}$, where the $v_i$ form a positively oriented orthonormal frame at $x$, be the frame bundle. The metric induces a geodesic flow $g^t:SM \rightarrow SM$, defined by $g^t(x,v) = (\gamma_{(x,v)}(t),\dot{\gamma}_{(x,v)}(t))$, where $\gamma_{(x,v)}$ is the unique geodesic determined by the vector $(x,v)$. The metric also induces a frame flow $F^t:FM \rightarrow FM$; defined by 
\[
F^t(x,v_0,v_1,\ldots,v_{n-1}) = (g^t(x,v_0),\Gamma_\gamma^t(v_1),\ldots,\Gamma_\gamma^t(v_{n-1})),
\]
where $\Gamma_\gamma^t$ denotes parallel transport along the geodesic $\gamma_{(x,v_0)}$. The frame bundle is a fiber bundle over $SM$, with structure group $SO(n-1)$ acting on the frames by rotations that keep the vector $v_0$ fixed. Hence, we have the following commuting diagram:
\[
\begin{tikzcd}
FM \arrow{r}{F^t} \arrow{d}{\pi} & FM \arrow{d}{\pi} \\
SM \arrow{r}{g^t} & SM
\end{tikzcd}
\]
The frame flow preserves a natural smooth measure $\mu = \mu_L \times \lambda_{SO(n-1)}$, where $\mu_L$ is (normalized) Liouville measure on the unit tangent bundle, and $\lambda_{SO(n-1)}$ is Haar measure on $SO(n-1)$. Note that $\pi_* \mu = \mu_L$, and $\mu_L$ is preserved by the geodesic flow.

\subsection{Partial hyperbolicity}
A flow $f^t: X \rightarrow X$ on a manifold $X$ with a Riemannian metric is called {\em partially hyperbolic} if the tangent bundle splits into three subbundles $TX = E^s \oplus E^c \oplus E^u$, each invariant under the flow, such that vectors in $E^s$ are (eventually) exponentially contracted by the flow, the vectors in $E^u$ are (eventually) exponentially expanded by the flow, and any contraction (resp. expansion) of vectors in $E^c$ is dominated by that of vectors in $E^s$ (resp. $E^u$). A flow is Anosov if these bundles can be chosen with $E^c=\langle \dot{f} \rangle$.

Some, but not necessarily all, distributions made up of these subbundles are integrable. For a point $x \in X$, the strong stable and strong unstable manifolds are defined by
\[
W^{su}(x) = \{ y \in X \, | \, d(f^{-t}x, f^{-t}y) \rightarrow 0 \text{ as } t \rightarrow +\infty \}
\]
\[
W^{ss}(x) = \{ y \in X \, | \, d(f^{t}x, f^{t}y) \rightarrow 0 \text{ as } t \rightarrow +\infty \}.
\]
These are $C^\infty$-immersed submanifolds of $X$, with $T_xW^{su}(x) = E^u(x)$ and $T_xW^{ss}(x) = E^s(x)$. The strong unstable (resp. stable) leaves form a foliation of $X$, which we denote $W^{su}$ (resp. $W^{ss}$).

If the flow is Anosov, the center bundle consists only of the flow direction and we can define weak-unstable (resp. weak-stable) manifolds through $x$ by
\[
W^u(x) = \bigcup_{t \in \mathbb{R}} W^{su}(f^tx)
\]
\[
W^s(x) = \bigcup_{t \in \mathbb{R}} W^{ss}(f^tx)
\]
Then $T_xW^u(x) = E^u(x) \oplus E^c(x)$ and $T_xW^s(x) = E^s(x) \oplus E^c(x)$. In contrast, the center bundle of a partially hyperbolic flow may be non-integrable, and the existence of weak-stable and weak-unstable manifolds is not guaranteed.

In our setting, the geodesic flow on a manifold of negative sectional curvature is an Anosov flow, while the frame flow is an example of a partially hyperbolic flow with center bundle of dimension $1 + \dim SO(n-1)$. The frame flow actually has a stronger property that implies partial hyperbolicity: it acts \textit{isometrically} on the center bundle, with respect to any bi-invariant metric on $SO(n-1)$.

\subsection{Pressure}
Given a function $\varphi: X \rightarrow \mathbb{R}$ (often called a \textit{potential}), consider the accumulation of $\varphi$ along orbits of $f^t$ given by $\varphi_T(x) = \int_0^T \varphi(f^t(x))\, dt$. Let $B(x,\epsilon) = \{ y \in X \, | \, d(x, y) < \epsilon \}$, and let 
\[
B(x,\epsilon,T) = \{ y \in X ~:~ d(f^t x, f^t y ) < \epsilon ~\text{for}~ 0 \leq t \leq T\}.
\]
If $\bigcup_{x \in E} B(x,\epsilon,T) = X$ for some set $E \subset X$, then $E$ is called $(T,\epsilon)$-spanning. Then the value 
\[
S(f,\varphi,\epsilon,T) = \inf \{ \sum_{x \in E} e^{\varphi_T(x)} ~:~ \bigcup_{x \in E} B(x,\epsilon,T) = X \}
\]
gives the minimum accumulation of $e^\varphi$ for time $T$ of a $(T,\epsilon)$-spanning set. The \textit{topological pressure} of $(f^t,\varphi)$ is the exponential growth rate (as $T\rightarrow \infty$) of this quantity as the resolution $\epsilon$ becomes finer:
\[
P(f^t,\varphi) = \lim_{\epsilon \to 0} \limsup_{T \to \infty} \frac{1}{T} \log S(f, \varphi, \epsilon, T).
\]

\nin Note that when $\varphi \equiv 0$, the sum $\sum e^{\varphi_T x}$ simply counts the elements of the $(T,\epsilon)$ spanning set, and we recover $h_{top}(f^t)$. The \textit{measure theoretic pressure}, with respect to an invariant measure $\mu$, is defined to be
\[
P_\mu(f^t,\varphi) = h_\mu(f^t) + \int_X \varphi \, d\mu.
\]
For a H\"older continuous potential $\varphi$, the variational principle states that
\[
P(f^t,\varphi) = \sup_{\mu \in \mathcal{M}(f^t)} P_\mu(f^t, \varphi).
\]
A measure $\mu \in \mathcal{M}(f^t)$ for which $P_\mu(f^t,\varphi) = P(f^t,\varphi)$ is called an \textit{equilibrium measure for $(f^t,\varphi)$}.

\subsection{Conditional measures}
Given a probability space $(Z,\mu)$, a measurable partition determines a way to disintegrate the measure $\mu$. Let $\mathcal{P}$ be a partition of $Z$ into measurable sets. Let $\pi: Z \rightarrow \mathcal{P}$ be the map sending $z\in Z$ to the element $Q \in \mathcal{P}$ that contains it, and set $\hat{\mu} := \pi_*\mu$ (note that this is a measure on the partition $\mathcal{P}$). Then a {\em system of conditional measures} (relative to $\mathcal{P}$) is a family $\{\mu_Q\}_{Q \in \mathcal{P}}$ such that
\begin{enumerate}
\item for each $Q \in \mathcal{P}$, the measure $\mu_Q$ is a probability measure on $Q$, and
\item for each $\mu$-measurable set $B \subset Z$, the map $Q \mapsto \mu_Q(B\cap Q)$ is $\hat{\mu}$-measurable and
\[
\mu(B) = \int_\mathcal{P} \mu_Q(B\cap Q) d\hat{\mu}(Q).
\]
\end{enumerate}
We will often abbreviate this second statement by writing $\mu = \int_\mathcal{P} \mu_Q d\hat{\mu}$. By a theorem of Rokhlin, whenever $\mathcal{P}$ is a measurable partition, there exists a system of conditional measures relative to $\mathcal{P}$. It is a straight-forward consequence of item (2) that any two such systems must agree on a set of full $\hat{\mu}$-measure.

%%%%%%%%%%%%%%%
% EQUILIBRIUM MEASURES FOR COMPACT GROUP EXTENSIONS
%%%%%%%%%%%%%%%
\section{Equilibrium measures for fiber bundles}
Let $\pi: Y \rightarrow X$ be a fiber bundle with $Y$ a measurable metric space and fibers a compact Lie group $V$. Let $F^t: Y \rightarrow Y$ be a smooth flow, and let $\mathcal{M}(F^t)$ denote the set of $F^t$-invariant probability measures on $Y$. If $F^t$ takes fibers to fibers and commutes with the action of the structure group (i.e., $F^t$ is a bundle automorphism for all $t \in \mathbb{R}$), then there is a map $f^t: X \rightarrow X$ such that $\pi \circ F^t = f^t \circ \pi$. In this case, an $F^t$-invariant probability measure $\mu \in \mathcal{M}(F^t)$ can be pushed forward to get an $f^t$-invariant probability measure $\hat{\mu} = \pi_*\mu \in \mathcal{M}(f^t)$ on $X$. Note that as long as the fibers are measurable sets, the partition $\{\pi^{-1}(x)\}_{x \in X}$ of $Y$ is measurable.

Let $\varphi: Y \rightarrow \mathbb{R}$ be a H\"older continuous function. Then, given a measure $\mu \in \mathcal{M}(F^t)$ and a disintegration of that measure $\mu = \int_X \mu_x d\hat{\mu}$, we can define a function $\hat{\varphi}_\mu: X \rightarrow \mathbb{R}$ by taking the average of $\varphi$ on the fibers $\pi^{-1}(x)$:
\[
\hat{\varphi}_\mu(x) = \int_{\pi^{-1}x} \varphi \, d\mu_x.
\]
Since any two such disintegrations of $\mu$ agree on a full $\hat{\mu}$ measure set, any two functions defined by different disintegrations of $\mu$ agree on a set of full $\hat{\mu}$ measure. In general, the function $\hat{\varphi}_\mu: X \rightarrow \mathbb{R}$ is only measurable, since the disintegration $x \mapsto \mu_x$ is only measurable. However, in the case that $\varphi$ is constant on fibers, the H\"older continuity of $\varphi$ implies that $\hat{\varphi}$ is also H\"older continuous.

While the existence of equilibrium measures is in general a non-trivial problem, work of Newhouse and Yomdin shows that such a measure always exists for $C^\infty$ dynamics. Since our Riemannian metric is $C^\infty$, both the geodesic flow and frame flow are also $C^\infty$, and so an equilibrium measure is guaranteed to exist \cite{newhouse89}. Thus, we only need to argue that such an equilibrium measure is unique.

Since $V$ is assumed to be a compact Lie group, it has a bi-invariant metric. We will say that $F: Y \rightarrow Y$ acts isometrically on the fibers if $F$ preserves distances in the fibers with respect to such a metric.

\begin{lemma}\label{embase}
Suppose that $F^t: Y \rightarrow Y$ acts isometrically on the fibers of the bundle $Y \rightarrow X$, and let $\varphi$ be a H\"older continuous function that is constant on fibers. Let $\mu$ be an equilibrium measure for $(F^t,\varphi)$. Then $\hat{\varphi}_\mu$ is a H\"older continuous function and $\hat{\mu}=\pi_*\mu$ is an equilibrium measure for $(f^t, \hat{\varphi}_\mu)$.
\end{lemma}

\begin{proof}
Since $F^t$ acts isometrically on the fibers, $F^t|_{\pi^{-1}(x)}$ does not generate any entropy. Then the Ledrappier-Walters formula~\cite{LW77} implies that $h_\mu(F^t) = h_{\hat{\mu}}(f^t)$, and
\begin{align*}
P(F^t,\varphi) &= h_\mu(F^t) + \int_Y \varphi \, d\mu \\
            &= h_{\hat{\mu}}(f^t) + \int_X \int_{\pi^{-1}x} \varphi \, d\mu_x d\hat{\mu}(x) \\
            &\leq P(f^t, \hat{\varphi}_\mu). 
\end{align*}
Now suppose that $\nu$ is an equilibrium measure on $X$ for $(f^t,\hat{\varphi}_\mu)$, and let $\lambda_x$ be (normalized) Haar measure on the fiber $\pi^{-1}(x)$. Then $x \mapsto \lambda_x$ is $\nu$-measurable, and $\tilde{\nu} = \int_X \lambda_x d\nu(x) \in \mathcal{M}(F^t)$. Moreover, $\hat{\varphi}_\mu(x) = \int \varphi \, d\mu_x = \int \varphi \, d\lambda_x$ because $\varphi$ is constant on fibers. Thus,
\begin{align*}
P(f^t, \hat{\varphi}_\mu) &= h_\nu(f^t) + \int_X (\int_{\pi^{-1}x} \varphi \, d\lambda_x) d\nu(x) \\
                        &= h_{\tilde{\nu}}(F^t) + \int_Y \varphi \, d\tilde{\nu} \\
                        &\leq P(F^t,\varphi),
\end{align*}
so $P(f^t,\hat{\varphi}_\mu) = P(F^t,\varphi)$. Hence,
\begin{align*}
P_{\hat{\mu}}(f^t,\hat{\varphi}_\mu) 
         &= h_{\hat{\mu}}(f^t) + \int_X \hat{\varphi}_\mu d\hat{\mu} \\
         &= h_\mu(F^t) + \int_Y \varphi \, d\mu \\
         &= P(F^t,\varphi) = P(f^t,\hat{\varphi}_\mu),
\end{align*}
so $\hat{\mu}$ is an equilibrium measure for $(f^t,\hat{\varphi}_\mu)$.
\end{proof}

In order to study the conditional measures $\mu_x$, we define subgroups of $\Isom V_x$ by their interaction with the measure $\mu_x$. Let $\mu$ be a measure on $Y$ with decomposition $\mu = \int \mu_x d\hat{\mu}$ along fibers, and let
\[
G_x^\mu = \{ \phi \in \Isom V_x ~|~ \phi_*\mu_x = \mu_x ~\text{and}~ \phi(\xi g) = \phi(\xi)g ~\text{for all}~ g \in V, \mu_x\text{-a.e.}~ \xi \in V_x \}.
\]
This is clearly a subgroup of $\Isom V_x$. Moreover, since any two decompositions of $\mu$ agree on a set of full $\hat{\mu}$-measure, any two sets $\{ G_x^\mu ~|~ x \in X\}$ defined using different decompositions of $\mu$ agree $\hat{\mu}$-almost everywhere.

The next lemma characterizes the support of conditional measures. It is an adaptation of Lemma 5.4 from~\cite{KS96}.

\begin{lemma}\label{clsubgrpsupp}
Let $F: Y \rightarrow Y$ be a fiber bundle automorphism that acts by isometries on the fibers and suppose $\mu$ is an ergodic measure. Then for $\hat{\mu}$-almost every $x \in X$ and $\mu_x$-almost every $\xi \in V_x$, 
\[
G_x^\mu \xi = \supp \mu_x.
\]
\end{lemma}

\begin{proof}
First, we show that $G_x^\mu \xi \subseteq \supp \mu_x$. Fix an $x \in X$, let $\phi \in G_x^\mu$ and let $\xi \in \supp \mu_x$. Let $B_\epsilon \subset V_x$ be the ball of radius $\epsilon$ around $\phi \, \xi$. Then, by the definition of $G_x^\mu$, we have $\phi_*\mu_x = \mu_x$ and
\[
\mu_x(B_\epsilon) = \phi_* \mu_x (B_\epsilon) = \mu_x (\phi^{-1}B_\epsilon) > 0
\]
because $\xi \in \phi^{-1}B_\epsilon$ and $\xi \in \supp \mu_x$. Hence, $\phi \, \xi \in \supp \mu_x$.

Next, we show that there is a set of full $\hat{\mu}$ measure in $X$ such that for $\mu_x$-almost every $\xi$, we have $G_x^\mu \xi \supseteq \supp \mu_x$. Let $\eta \in \supp \mu_x$; we will show that for $\mu_x$-almost every $\xi$ there is a $\phi \in G_x^\mu$ such that $\phi (\xi) = \eta$. Recall that the $\mu$-disintegration $x \mapsto \mu_x$ is only a measurable map. By Lusin's Theorem, however, for any $\epsilon>0$ there exists a closed set $K_\epsilon \subset X$ such that
\begin{enumerate}
\item the map $K_\epsilon \rightarrow \mathcal{M}(E)$ taking $x$ to $\mu_x$ is continuous, and
\item $\hat{\mu}(K_\epsilon) > 1-\epsilon$.
\end{enumerate}
Since $\mu$ is an ergodic measure, the Birkhoff Ergodic Theorem implies that $\mu$-almost every point in $Y$ has dense orbit in $\supp \mu$. Thus, the set
\[
K'_\epsilon := \{ \xi \in K_\epsilon \, | \, \mu_x\text{-a.e.} \, \xi \in \pi^{-1}(x) \text{ has dense orbit in } \supp \mu \}
\]
has measure $\hat{\mu}(K'_\epsilon) = \hat{\mu}(K_\epsilon) > 1-\epsilon$. Moreover, for any $x \in K'_\epsilon$ and $\mu_x$-almost any $\xi \in \pi^{-1}(x)$, there exists a sequence of times $t_i$ such that $F^{t_i} \xi \rightarrow \eta$. Then:
\begin{enumerate}
\item Since $F^{t_i}|_{\pi^{-1}(x)}$ are all isometries, $F^{t_i}|_{\pi^{-1}(x)} \rightarrow \phi \in \Isom V_x$.
\item $\phi(\xi) = \eta$.
\item By the $F^t$-invariance of $\mu$, we have that $F^t_* \mu_x = \mu_{f^t x}$ for $\hat{\mu}$-almost every $x$. This, together with the continuity of the map $x \mapsto \mu_x$ on $K'_\epsilon$, implies that
\[
\phi_* \mu_x = \lim_{i \to \infty} F^{t_i}_* \mu_x = \lim_{i \to \infty} \mu_{f^{t_i} x} = \mu_x.
\]
Further, since $F^t$ commutes with the action of the isometry group on the fibers and $F^{t_i}|_{\pi^{-1}(x)} \rightarrow \phi$, we get that $\phi(\xi g) = \phi(\xi)g$. Hence, $\phi \in G_x^\mu$.
\end{enumerate}
Letting $\epsilon \to 0$ gives a set $K$ of full $\hat{\mu}$-measure such that for any $x \in K$, $\mu_x$-almost every $\xi$ satisfies $G_x^\mu \xi \supseteq \supp \mu_x$.
\end{proof}

\subsection{Regularity of sections}
% local product structure
Let $M$ be a closed manifold, and consider an Anosov flow $f^t: M \rightarrow M$. Then $M$ has the following local product structure with respect to the strong stable foliation $W^{ss}$ and weak unstable foliation $W^{u}$: for any $x \in M$, there exists a neighborhood $V_x$ of $x$ such that every point in $V_x$ can be written as $[y,z] := W_{loc}^u(y) \cap W_{loc}^{ss}(z)$ for some $y \in W_{loc}^{ss}(x)$ and $z \in W_{loc}^u(x)$. 

Let $\pi_x^\sigma: V_x \rightarrow W_{loc}^\sigma(x)$ (for $\sigma=ss,u$) be the projection map onto the appropriate local manifold. Given a measure $\mu$ on $M$, let $\mu_x^\sigma=(\pi_x^\sigma)_* \mu$ (for $\sigma=ss,u$). We say that \emph{$\mu$ has local product structure} if 
\[ 
d \mu ([y,z]) =  \phi_x(y,z) \: d \mu_x^{ss} (y)  \times d \mu_x^u (z) 
\]
for any $y \in W_{loc}^{ss}(x)$ and any $z \in W_{loc}^u(x)$, where $\phi_x$ is a non-negative Borel function.

\begin{remark}
The definition above corresponds to the one used in~\cite{Leplaideur}. An alternate definition states that $\mu$ has local product structure if $\mu$ is \textit{locally equivalent} to $\mu_x^{ss} \times \mu_x^u$~\cite{Bonatti-Diaz-Viana}. This alternate definition is a strictly stronger property, since the function $\phi_x$ above is only a \textit{non-negative} Borel function.
\end{remark}

Methods of \cite{GS97} extend to prove the following Liv\v{s}ic regularity theorem. 

\begin{theorem}\label{sectionregularity}
Let $P \rightarrow M$ be a principal $H$-bundle over a compact, connected manifold $M$ with $H$ a compact group. Suppose $G=\mathbb{R}$ or $\mathbb{Z}$ acts H\"older continuously by bundle automorphisms such that the induced action on $M$ is Anosov. Let $m$ be a $G$-invariant measure on $M$ with local product structure, and let $L \subseteq M$. Let $V$ be a transitive left $H$-space that admits an $H$-invariant metric, and consider the associated bundle $E_V \rightarrow M$. Then any $G$-invariant measurable (w.r.t. $m$) section $L \rightarrow E_V$ is H\"older continuous on a subset $L' \subseteq L$ with $m(L') = m(L)$.
\end{theorem}

\begin{remark}
Goetze and Spatzier in \cite{GS97} prove this result for Anosov actions for an \textit{invariant smooth measure} for Lie groups with a bi-invariant metric; the theorem above extends this to invariant measures with local product structure.
\end{remark}

\begin{remark}\label{origbdle}
In particular, $H$ is a transitive left $H$-space with $H$-invariant metric. In this case, the conclusion of the theorem states that any $G$-invariant measurable section $L \rightarrow P$ is H\"older continuous on a full measure subset of $L$.
\end{remark}

The remainder of this section provides an outline of the ideas used to prove this theorem in the case $V=H$; the interested reader can find details in \cite{GS97}. First, we discuss the relationship between sections of a bundle and cocycles. A \textit{cocycle} is a map $\alpha: G\times M\rightarrow H$ such that 
\[
\alpha(t_2+t_1, x) = \alpha(t_2, g_{t_1} x) \alpha(t_1,x).
\]
Two such cocycles $\alpha$ and $\beta$ are \textit{cohomologous} if there is a function $\psi: M \rightarrow H$ such that
\[
\alpha(t,x) = \psi(g_t x) \beta(t,x) \psi(x)^{-1}.
\]
We say that $\alpha$ and $\beta$ are \textit{measurably} (resp. \textit{H\"older}) \textit{cohomologous} if $\psi$ can be chosen to be a measurable (resp. H\"older continuous) function. If the underlying action is ergodic, then it follows that $\psi$ is unique up to a constant. In this case, the regularity of the cohomology doesn't depend on the $\psi$ chosen.

Measurable sections correspond under the $G$-action to measurable cocycles as follows. Let $\sigma: M \rightarrow P$ be a measurable section. This determines a measurable cocycle $\alpha: G \times M \rightarrow H$ by the relationship 
\[
g_t \sigma(x) = \sigma(g_t x) \alpha(t,x),
\]
($\alpha$ is uniquely determined since $P \rightarrow M$ is a principle $H$-bundle). Given a function $b: M \rightarrow H$, the section $\sigma_b(x) := \sigma(x) b(x)$ yields a cocycle $\beta(t,x)$ that is cohomologous to $\alpha$. As $\sigma$ is a measurable $G$-invariant section, $\alpha$ is in fact a measurable coboundary (i.e., measurably cohomologous to the trivial cocycle).

Since $P \rightarrow M$ may be a non-trivial bundle, there may be topological obstructions to a continuous section $M \rightarrow P$. Thus, in order to consider increasing the regularity of the section $\sigma$, we must break up the correspondence with cocycles into local trivializations on open sets $U_i$ that cover $M$. On an open set $U_i$, there is a smooth section $s_i: U_i \rightarrow P$. Then, for any $x \in U_i$, there is a map $h_i: P \rightarrow H$ such that for any $x^*\in P$ in the fiber over $x \in U_i$,
\[
x^* = s_i(x) h_i(x^*).
\]
In particular, $\sigma(x) = s_i(x) h_i(\sigma(x))$. Note that the maps $h_i$ are uniformly Lipschitz, but $h_i \circ \sigma$ is \textit{a priori} only measurable since $\sigma$ is only measurable.

The proof of the theorem thus reduces to showing that $h_{i(x)} \circ \sigma: L \rightarrow H$ is H\"older continuous. (This then implies that the section $\sigma$ must also be H\"older continuous, since $h_{i(x)}$ is uniformly Lipschitz and its image is transverse to the fibers.) Let $L' \subseteq L$ be the set of points that are also in the support of $m$ and for which the Birkhoff Ergodic Theorem holds. We have $m(L') = m(L)$. Consider two points $x, w \in L'$ that are close enough to be in a neighborhood of local product structure. Then, there exist $y \in W_{loc}^{ss}(x)$ and $z \in W_{loc}^u (x)$ such that $w=[y,z]_x$. Observe that also $x=[z,y]_w$.

By local product structure of $m$, we have 
\[ 
d m ([y,z]) =  \phi_x(y,z) \: d m_x^{ss} (y)  \times d m_x^u (z) 
\]
for a non-negative Borel function $\phi_x$. Since $w=[y,z]$ is in the support of $m$, $\phi$ must be positive at $(y,z)$, $y$ must be in the support of $m_x^{ss}$, and $z$ must be in the support of $m_x^u$. Likewise, $z$ must be in the support of $m_w^{ss}$ and $y$ must be in the support of $m_w^{u}$. Then, by the triangle inequality,
\[
d(h_i(\sigma(x)), h_i(\sigma(w)) \leq d(h_i(\sigma(x)),h_i(\sigma(y))) + d(h_i(\sigma(y)),h_i(\sigma(w)))
\]
(where $d$ denotes distance in $H$). Thus, the problem is further reduced to showing that $h_{i(x)} \circ \sigma$ is H\"older continuous along stable and unstable manifolds.

Toward this end, consider two points $x,y \in L'$ on the same stable manifold. We want to measure the distance between $h_{i(x)}(\sigma(x))$ and $h_{i(y)}(\sigma(y))$. Recall that $\sigma$ is a $G$-invariant section, and note that $x$ and $y$ are on the same stable manifold and so eventually are in the same open set $U_j$. Then the H\"older continuity of the $G$-action, along with the exponential contraction along the stable leaf, allow us to reduce our consideration to the distance between $h_j(g_t \sigma(x))$ and $h_j(g_t \sigma(y))$. 

Although $\sigma$ is a measurable section, Lusin's Theorem guarantees a compact set $K \subset L$ with $m(K)>1/2$ on which $\sigma$ is uniformly continuous. This implies that, for $x$ and $y$ in a set of full measure, there is an unbounded set of $t$ such that $g_t \sigma(x)$ and $g_t \sigma(y)$ are in $K$. For such $x$ and $y$, combining this with the previous paragraph gives H\"older continuity of $h_{i(x)} \circ \sigma$ along the local stable manifold. A similar argument shows that $h_{i(x)} \circ \sigma$ is H\"older continuous along the unstable manifold.

%%%%%%%%%%%%%%%
% E.M. FOR FRAME FLOW 
%%%%%%%%%%%%%%%
\section{Equilibrium measures for frame flows}
Consider the frame flow $F^t: FM \rightarrow FM$ of a closed, oriented, negatively curved manifold $M$, and a H\"older continuous potential $\varphi: FM \rightarrow \mathbb{R}$ that is constant on the fibers of the bundle $FM \rightarrow SM$. By smoothness assumptions, there exists an equilibrium measure $\mu$ for $(F^t,\varphi)$. Then, by Lemma~\ref{embase}, $\hat{\mu}$ is an equilibrium measure for $(f^t, \hat{\varphi}_\mu)$, and $\hat{\varphi}_\mu$ is H\"older continuous. Next, we recall some results about equilibrium measures for hyperbolic flows:

\begin{theorem}[Bowen-Ruelle, Leplaideur, \cite{Bowen-Ruelle}, \cite{Leplaideur}]
Let $f^t: M \rightarrow M$ be an Anosov flow and $\varphi: M \rightarrow \mathbb{R}$ a H\"older continuous potential. Then there exists a unique equilibrium measure for $(f^t,\varphi)$. It is ergodic and has local product structure and full support.
\end{theorem}

\subsection{Proof of Theorem~\ref{main} in dimension 3}
We first prove Theorem~\ref{main} in the case $n=3$, where the logic of the proof is the same but the groups are simpler. The reason for this is that in the $3$-dimensional case, the fibers of the bundle $FM \rightarrow SM$ are $\mathbb{S}^1$, which is an abelian group.

Consider the conditional measures $\{ \mu_x \}$ given by Rokhlin decomposition $\mu = \int \mu_x d\hat{\mu}$. By Theorem~\ref{clsubgrpsupp}, the support of a conditional measure $\mu_x$ for a typical point $x$ is the orbit of a closed subgroup of isometries of the fiber. Since the fibers of $FM \rightarrow SM$ are $SO(2) \approx \mathbb{S}^1$, we get two possibilities: either
\begin{enumerate}
\item $\supp \mu_x = \mathbb{S}^1$ for $\hat{\mu}$-almost every $x$, or
\item $\mu_x$ is atomic and supported on $m$ points for $\hat{\mu}$-almost every $x$.
\end{enumerate}
In the first case, we are done, since $\hat{\mu}$ is unique and thus $\mu = \int \mu_x \, d\hat{\mu}$ is uniquely determined. We will show that topological considerations prevent the second case from occurring.

The easiest case is if $m=1$. Let $L$ be the full $\hat{\mu}$-measure set on which $\supp \mu_x$ is one point. This gives a measurable section $\sigma: L \longrightarrow FM$ sending $x$ to the point $\supp \mu_x$. By Theorem~\ref{sectionregularity}, $\sigma$ is actually H\"older continuous. Because $\hat{\mu}(L)=1$ and $\hat{\mu}$ has full support on $SM$, $\sigma$ can then be extended to a H\"older continuous section $SM \rightarrow FM$. Restricting this section to $S_p M$ for some $p \in M$ gives a continuous map on $S_p M \approx \mathbb{S}^2$ that sends each point to an element of $SO(2) \approx \mathbb{S}^1$. This can be seen as a non-vanishing, continuous vector field on $\mathbb{S}^2$, which is a contradiction.

Now suppose that $m>1$. Let $F$ be the fiber of the bundle $FM \to SM$. Construct a new bundle $FM^m \to SM$ with fibers the Cartesian product $F^m$. Then the discussion above produces a measurable map $\sigma: L \rightarrow FM^m/\Sigma_m$, where $\Sigma_m$ acts on $FM^m$ by permutations, sending $x \in L$ to the $m$ points in the support of $\mu_x$. Now we apply Theorem~\ref{sectionregularity} to the associated bundle $E_V = FM^m/\Sigma_m$ with $H = SO(2)^m$ and $V = SO(2)^m/\Sigma_m$. This implies that $\sigma$ is a H\"older continuous map on $L$, which then can be extended to a H\"older continuous map $M \rightarrow FM^m/\Sigma_m$. Then the projection map $FM^m/\Sigma_m \rightarrow SM$, restricted to the preimage of a set $S_pM$ is an $m$-fold cover of $S_pM \approx \mathbb{S}^2$. Since $\mathbb{S}^2$ is simply connected, the preimage must be a disjoint union of $m$ copies of $\mathbb{S}^2$. Restricting the cover to one of these copies gives a non-vanishing, continuous vector field on $\mathbb{S}^2$ as in the case $m=1$, which is a contradiction.

\subsection{Proof of Theorem~\ref{main} in higher dimensions} \label{sec:ProofThm1}
More generally, the structure group of $FM \to SM$ is $SO(n-1)$. Recall that the measure $\mu_x$ is invariant under the group $G_x^\mu$ by construction. Then either $G_x^\mu = SO(n-1)$ for $\hat{\mu}$-almost every $x$ and $\mu$ is determined by $\hat{\mu}$, or else $G_x^\mu$ is a strict subgroup of $SO(n-1)$ $\hat{\mu}$-almost everywhere. Moreover, by ergodicity of $\hat{\mu}$, $G_x^\mu$ must be the same subgroup $H < SO(n-1)$ $\hat{\mu}$-almost everywhere.

Thus, we get a measurable section $\sigma: SM \rightarrow FM/H$ that takes $x$ to the support of $\mu_x$ on a set of full $\hat{\mu}$ measure. By Theorem~\ref{sectionregularity}, we can extend this to a continuous, global section $SM \rightarrow FM/H$. Such a section gives a reduction of the structure group of $FM \rightarrow SM$, as follows. The section $\sigma$ provides a trivialization of the bundle $FM/H \rightarrow SM$. Since $FM \rightarrow FM/H$ is a principle $H$-bundle, the the pullback $\sigma^*(FM) = FM_H$ is a reduction of $FM$ with structure group $H$.
\[
\begin{tikzcd}
FM_H \arrow[dotted]{r}{} \arrow[dotted]{d}{} & FM \arrow{d}{\pi} \\
SM \arrow{r}{\sigma} & FM/H
\end{tikzcd}
\]

A non-trivial reduction of the structure group of $FM \rightarrow SM$ also gives a non-trivial reduction of the structure group of the restricted bundle $FM \rightarrow S_pM \approx \mathbb{S}^{n-1}$. Then, for $n$ odd and not equal to $7$, Proposition 5.1 of~\cite{Brin-Gromov} implies that $H$ cannot act transitively on $\mathbb{S}^{n-2} = \pi_2^{-1}(p)$. However, for $n$ odd, any structure group must act transitively on $\mathbb{S}^{n-2}$, by Corollary 4.2 of~\cite{Brin-Gromov}. Thus, $G_x^\nu$ must be equal to $SO(n-1)$ $\mu_L$-a.e., and hence by Lemma~\ref{erglemma} the frame flow must be ergodic. \qed

\subsection{Proof of Theorem~\ref{heisenberg}} 

%The following gives a proof of Theorem~\ref{heisenberg}. 
The bundle $M \to \mathbb{T}^2$ is a fiber bundle with structure group $\mathbb{S}^1$.  Any automorphism of $M$ automatically preserves the $\mathbb{S}^1$ fibers, acts by volume preserving diffeomorphisms on both $M$ and $\mathbb{T}^2$, and hence has eigenvalue 1 in the (invariant) fiber direction. Thus $f$ acts by isometries on the fibers.  By assumption $f$ induces an Anosov diffeomorphism on the base torus $\mathbb{T}^2$. 
Let $\varphi: M \rightarrow \mathbb{R}$ be a H\"older continuous potential function that is constant on the fibers of $M \to \mathbb{T}^2$, and let $m$ be an equilibrium measure for $(f,\varphi)$. 
Then the arguments from the case of frame flows apply verbatim and give us the dichotomy: either $m$ is invariant under the structure group $\mathbb{S}^1$, and we can understand the equilibrium state via the base torus $\mathbb{T}^2$, or there is a continuous invariant section of the fibre bundle $M \to \mathbb{T}^2$ or a finite cover. In the first case, $m$ is uniquely determined since the equilibrium measure of an Anosov diffeomorphism on $\mathbb{T}^2$ is unique (Theorem 4.1 of~\cite{Bowen}). The second case implies that for a finite cover $\bar{M}$ of $M$, $\pi _1 (\bar{M}) = \mathbb{Z} \times \mathbb{Z} ^2$, which is a impossible since $Heis(\mathbb{Z})$ does not contain a rank $3$ abelian subgroup.

%%%%%%%%%%%%%%%
% ERGODICITY OF FRAME FLOW 
%%%%%%%%%%%%%%%
\section{Ergodicity of the frame flow, revisited}\label{sec:reproof}

In this section, we revisit a theorem of Brin and Gromov on the ergodicity of frame flows on negatively curved manifolds in odd dimensions~\cite{Brin-Gromov}. Recall that there is a natural smooth invariant measure $\mu$ for the frame flow on the frame bundle which projects to the Liouville measure $\mu_L$ on the unit tangent bundle and is invariant under the structure group $SO(n-1)$ for the bundle $FM \rightarrow SM$. 

\begin{theorem}[Brin-Gromov, \cite{Brin-Gromov}] \label{thm:BGsec5}
Let $M$ be an odd dimensional closed, oriented manifold of negative sectional curvature and dimension $n \neq 7$. Then the frame flow is ergodic.  
\end{theorem}

First, we characterize the ergodicity of the frame flow in terms of the groups $G_x^\mu$.

\begin{lemma}\label{erglemma} % cf. Brin82 4.6
The frame flow $F^t: FM \rightarrow FM$ is ergodic if and only if for any ergodic component $\nu$ of $\mu$, $G_x^\nu = SO(n-1)$ for $\mu_L$-a.e. $x \in SM$.
\end{lemma}

\begin{proof} 
Consider an ergodic component $\nu$ of $\mu$ for the frame flow, and suppose $G_x^\nu = SO(n-1)$ for $\mu_L$-almost every $x \in SM$. Since the Liouville measure $\mu_L$ on $SM$ is ergodic for the geodesic flow and $\mu$ projects to $\mu_L$, the measure $\nu$ also projects to $\mu_L$. By definition of $G_x^\nu$, each conditional measure $\nu_x$ is invariant under $G_x^\nu$. Since $G_x^\nu = SO(n-1)$ for $\mu_L$-a.e. $x$, the measure $\nu_x$ must be Haar measure for $\mu_L$-a.e. $x$. Thus, $\nu = \mu$.

Conversely, if $\mu$ is ergodic for the frame flow, then $G_x^\mu$ acts transitively on $\pi^{-1}(x) \approx SO(n-1)$ for $\mu_L$-a.e. $x$ by Lemma~\ref{clsubgrpsupp}. Since $G^\mu_x$ commutes with the right action of $SO(n-1)$, this action is also free, since any element $\alpha \in G^\mu_x$ that fixes the identity $e$ satisfies
\[
\alpha(g) = \alpha(e\cdot g) = \alpha(e) \cdot g = e \cdot g = g.
\]
Hence, $G_x^\mu = SO(n-1)$ for $\mu_L$-a.e. $x$.
\end{proof}

Thus, to prove Theorem~\ref{thm:BGsec5} we need only show that, for an arbitrary ergodic component $\nu$ of $\mu$, $G_x^\nu$ must be $SO(n-1)$. Toward that end, suppose that $G_x^\nu \not= SO(n-1)$. Then the following lemma shows that there is a non-trivial reduction of the structure group $SO(n-1)$.

\begin{lemma}
Let $\nu$ be an ergodic component of $\mu$. If $G^\nu_x$ is a proper subgroup of $SO(n-1)$ $\mu_L$-a.e., then there is a non-trivial reduction of the structure group of $FM \rightarrow SM$. 
\end{lemma}

\begin{proof}
We find a reduction of the structure group by utilizing a natural one-to-one correspondence between reductions of the structure group $SO(n-1)$ to a subgroup $H$ and continuous global sections of the fiber bundle $FM/H \rightarrow SM$ (described in Section~\ref{sec:ProofThm1}).

We need to produce such a subgroup $H$. The obvious candidate is $G_x^\nu$, but a priori this group depends on $x$; we will show that $G_x^\nu$ is the same subgroup for a set of full $\hat{\mu}$ measure. Along an orbit of the geodesic flow, we have $G_{g^t x}^\nu = F^t G_x^\nu F^{-t}$. Since the $SO(n-1)$ action on the fibers commutes with the frame flow, this identification gives the same group along an orbit. By ergodicity of the geodesic flow, $G_x^\nu$ must be the same group inside $SO(n-1)$ $\mu_L$-almost everywhere. This gives a measurable section $\sigma: SM \rightarrow FM/H$ taking $x$ to the support of $\mu_x$ on this set of full $\mu_L$ measure. By Theorem~\ref{sectionregularity}, we can extend this to a continuous, global section $SM \rightarrow FM/H$. By the correspondence described above, this gives a reduction of the structure group to $H$, which by assumption is a proper subgroup of $SO(n-1)$.
\end{proof}

A non-trivial reduction of the structure group of $FM \rightarrow SM$ also gives a non-trivial reduction of the structure group of the restricted bundle $FM \rightarrow S_pM \approx \mathbb{S}^{n-1}$. By arguments of Brin and Gromov~\cite{Brin-Gromov}, this implies that $G_x^\nu$ must be equal to $SO(n-1)$ $\mu_L$-a.e. Hence, by Lemma~\ref{erglemma}, the frame flow is ergodic.

\bibliographystyle{abbrv}
\bibliography{EMIEAS.bib}

\begin{thebibliography}{10}

\bibitem{AVW}
A.~Avila, M.~Viana, and A.~Wilkinson.
\newblock Absolute continuity, {L}yapunov exponents and rigidity {I}: geodesic
  flows.
\newblock {\em J. Eur. Math. Soc. (JEMS)}, 17(6):1435--1462, 2015.

\bibitem{Bonatti-Diaz-Viana}
C.~Bonatti, L.~J. D{\'{\i}}az, and M.~Viana.
\newblock {\em Dynamics beyond uniform hyperbolicity}, volume 102 of {\em
  Encyclopaedia of Mathematical Sciences}.
\newblock Springer-Verlag, Berlin, 2005.
\newblock A global geometric and probabilistic perspective, Mathematical
  Physics, III.

\bibitem{Bowen}
R.~Bowen.
\newblock {\em Equilibrium states and the ergodic theory of {A}nosov
  diffeomorphisms}.
\newblock Lecture Notes in Mathematics, Vol. 470. Springer-Verlag, Berlin-New
  York, 1975.

\bibitem{Bowen-Ruelle}
R.~Bowen and D.~Ruelle.
\newblock The ergodic theory of {A}xiom {A} flows.
\newblock {\em Invent. Math.}, 29(3):181--202, 1975.

\bibitem{Brin-Gromov}
M.~Brin and M.~Gromov.
\newblock On the ergodicity of frame flows.
\newblock {\em Invent. Math.}, 60(1):1--7, 1980.

\bibitem{Brin-Karcher}
M.~Brin and H.~Karcher.
\newblock Frame flows on manifolds with pinched negative curvature.
\newblock {\em Compositio Math.}, 52(3):275--297, 1984.

\bibitem{Burns-Pollicott}
K.~Burns and M.~Pollicott.
\newblock Stable ergodicity and frame flows.
\newblock {\em Geom. Dedicata}, 98:189--210, 2003.

\bibitem{Climenhaga-Fisher-Thompson}
V.~Climenhaga, T.~Fisher, and D.~Thompson.
\newblock Unique equilibrium states for some robustly transitive systems.
\newblock Preprint. arXiv:1505.06371.

\bibitem{Climenhaga-Thompson}
V.~Climenhaga and D.~Thompson.
\newblock Unique equilibrium states for flows and homeomorphisms with
  non-uniform structure.
\newblock Preprint. arXiv:1505.03803.

\bibitem{EKL06}
M.~Einsiedler, A.~Katok, and E.~Lindenstrauss.
\newblock Invariant measures and the set of exceptions to {L}ittlewood's
  conjecture.
\newblock {\em Ann. of Math. (2)}, 164(2):513--560, 2006.

\bibitem{EL15}
M.~Einsiedler and E.~Lindenstrauss.
\newblock On measures invariant under tori on quotients of semisimple groups.
\newblock {\em Ann. of Math. (2)}, 181(3):993--1031, 2015.

\bibitem{GS97}
E.~R. Goetze and R.~J. Spatzier.
\newblock On {L}iv\v sic's theorem, superrigidity, and {A}nosov actions of
  semisimple {L}ie groups.
\newblock {\em Duke Math. J.}, 88(1):1--27, 1997.

\bibitem{KS96}
A.~Katok and R.~J. Spatzier.
\newblock Invariant measures for higher-rank hyperbolic abelian actions.
\newblock {\em Ergodic Theory Dynam. Systems}, 16(4):751--778, 1996.

\bibitem{Knieper}
G.~Knieper.
\newblock The uniqueness of the measure of maximal entropy for geodesic flows
  on rank {$1$} manifolds.
\newblock {\em Ann. of Math. (2)}, 148(1):291--314, 1998.

\bibitem{Knieper05}
G.~Knieper.
\newblock The uniqueness of the maximal measure for geodesic flows on symmetric
  spaces of higher rank.
\newblock {\em Israel J. Math.}, 149:171--183, 2005.
\newblock Probability in mathematics.

\bibitem{LW77}
F.~Ledrappier and P.~Walters.
\newblock A relativised variational principle for continuous transformations.
\newblock {\em J. London Math. Soc. (2)}, 16(3):568--576, 1977.

\bibitem{Leplaideur}
R.~Leplaideur.
\newblock Local product structure for equilibrium states.
\newblock {\em Trans. Amer. Math. Soc.}, 352(4):1889--1912, 2000.

\bibitem{LS04}
E.~Lindenstrauss and K.~Schmidt.
\newblock Invariant sets and measures of nonexpansive group automorphisms.
\newblock {\em Israel J. Math.}, 144:29--60, 2004.

\bibitem{LS05}
E.~Lindenstrauss and K.~Schmidt.
\newblock Invariant sets and measures of nonexpansive group automorphisms.
\newblock {\em Israel J. Math.}, 144:29--60, 2004.

\bibitem{livsic}
A.~N. Liv{\v{s}}ic.
\newblock Cohomology of dynamical systems.
\newblock {\em Izv. Akad. Nauk SSSR Ser. Mat.}, 36:1296--1320, 1972.

\bibitem{newhouse89}
S.~E. Newhouse.
\newblock Continuity properties of entropy.
\newblock {\em Ann. of Math. (2)}, 129(2):215--235, 1989.

\bibitem{RHRHTU12}
F.~Rodriguez~Hertz, M.~A. Rodriguez~Hertz, A.~Tahzibi, and R.~Ures.
\newblock Maximizing measures for partially hyperbolic systems with compact
  center leaves.
\newblock {\em Ergodic Theory Dynam. Systems}, 32(2):825--839, 2012.

\bibitem{Sun2012}
P.~Sun.
\newblock Density of metric entropies for linear toral automorphisms.
\newblock {\em Dyn. Syst.}, 27(2):197--204, 2012.

\bibitem{Ures}
R.~Ures.
\newblock Intrinsic ergodicity of partially hyperbolic diffeomorphisms with a
  hyperbolic linear part.
\newblock {\em Proc. Amer. Math. Soc.}, 140(6):1973--1985, 2012.

\end{thebibliography}

\end{document}